\documentclass{article}

\usepackage[english]{babel}

\usepackage[letterpaper,top=2cm,bottom=2cm,left=3cm,right=3cm,marginparwidth=1.75cm]{geometry}

\usepackage{amsmath}
\usepackage{graphicx}
\usepackage[colorlinks=true, linkcolor=black, allcolors=blue]{hyperref}
\usepackage{amssymb,amsthm}
\usepackage{amsfonts,graphicx,psfrag}
\usepackage{graphicx}
\usepackage{float}
\usepackage{subcaption}
\usepackage{appendix}

\newcommand{\fe}{\mathfrak{e}}
\newcommand{\M}{\mathfrak{M}}

\newcommand{\rr}{\mathbb{R}}

\newtheorem{thm}{Theorem}[section]
\newtheorem{defi}{Definition}[section]
\newtheorem{prop}{Proposition}[section]
\newtheorem{lem}{Lemma}[section]
\newtheorem{cor}{Corollary}[section]
\newtheorem{rem}{Remark}[section]

\def\bea{\begin{eqnarray}}
\def\eea{\end{eqnarray}}
\def\no{\noindent}
\def\bs{\bigskip}
\def\Sp{{\mathrm {Sp}}}

\begin{document}
\title{Relative Periodic Orbits in the Spatial Anisotropic Kepler Problem}
\author{Xijun Hu\thanks{e-mail:xjhu@sdu.edu.cn}
\quad Yuwei Ou\thanks{e-mail:ywou@sdu.edu.cn}
\quad Zhiwen Qiao \thanks{e-mail:qiaozw@mail.sdu.edu.cn.}
\\ \\
School of Mathematics, Shandong University
Jinan, Shandong 250100\\
The People's Republic of China
}

\date{}
\maketitle
\begin{abstract}
The spatial anisotropic Kepler problem comes from quantum mechanics, which models electron motion in semiconductors with donor impurities and depends on the anisotropic parameter
$\beta\in(-1,+\infty)$. After reducing the system modulo rotational symmetry, we investigate periodic orbits in this two-degree-of-freedom setting. Combining an index comparison in \cite{HLOQS26}, the volume formula in \cite{CHHL23} and Franks' theorem, we prove that the system possesses infinitely many periodic orbits on any fixed compact and regular energy surface for
$\beta\in(-1,0]$.

\end{abstract}

\bs

\no{\bf AMS Subject Classification:} 70F10, 37J46, 53D12, 53D10

\bs

\no{\bf Key Words: $N$-body problem, periodic solution, contact manifold, Maslov-type index}


\tableofcontents
\section{Introduction and main results}
	
	

The spatial anisotropic Kepler problem is a modified model of the spatial Kepler problem and originally came from quantum mechanics. It was first considered by Gutzwiller, see \cite{GMC73,GMC1990}, and it models the motion of an electron in semiconductors with donor impurities in $\mathbb R^3$. In Cartesian coordinates $q=(x,y,z)\in\mathbb{R}^{3}$ with the corresponding momentum variable $p=(p_x,p_y,p_z)=(\dot{x},\dot{y},\dot{z})\in\mathbb{R}^{3}$, the Hamiltonian of the spatial anisotropic Kepler problem can be written as
\begin{equation}
	H(p,q)=\frac{1}{2}(p_x^2+p_y^2+p_z^2)-\frac{1}{\sqrt{x^2+y^2+(1+\beta) z^2}},\nonumber
\end{equation}
where $\beta>-1$ is the anisotropic parameter. If $\beta=0$, we get the well-known Hamiltonian of the spatial Kepler problem.
Guirao, Llibre and Vera \cite{GLV2013} established the existence of periodic orbits in the spatial anisotropic Kepler problem as a special case of the perturbed Kepler problem. Please refer to \cite{LL21,LM2012} for more general and relevant studies. In higher-dimensional cases, Barutello, Terracini and Verzini \cite{BTG14} studied the parabolic trajectories of the anisotropic Kepler problem in $\mathbb R^d$ with prescribed asymptotic directions at infinity. In \cite{HY18}, Hu and Yu developed an index theory for zero-energy solutions of the planar anisotropic Kepler problem and established the relations between the Morse indices of zero-energy solutions and their oscillatory behaviors. Yu \cite{Y25} established the asymptotic properties for the positive-energy solutions of the anisotropic Kepler problem in $\mathbb R^d$ with a homogeneous potential, and proved certain existence results of hyperbolic and bi-hyperbolic solutions.

Since this problem is symmetric with respect to $z$-axis, we can use cylindrical coordinates $(x,y,z)=(r\cos\theta,r\sin\theta,z)$ with the corresponding momentum $(p_r,p_{\theta},p_z)=(\dot{r},r^2\dot{\theta},\dot{z})$; the above Hamiltonian can be written correspondingly as,
\begin{equation*}
	H(p_r,p_{\theta},p_z,r,\theta,z)=\frac{1}{2}(p_r^2+\frac{p_{\theta}^2}{r^2}+p_z^2)-\frac{1}{\sqrt{r^2+(1+\beta) z^2}}.
\end{equation*}
Because the system is symmetric with respect to $z$-axis, $p_{\theta}$ is the angular momentum with respect to the $z$-axis, which is a first integral. Without loss of generality, we fix the angular momentum $p_{\theta}=\varpi>0$. Then we get the Hamiltonian of the reduced spatial anisotropic Kepler problem,
\begin{equation}\label{eq;reducedHam}
	H_{\varpi}(p_r, p_z, r, z)=\frac{1}{2}(p_r^2+p_z^2)+V_{\varpi}(r,z),\ \text{with} \; V_{\varpi}(r, z) =  \frac{\varpi^2}{2r^2 } - \frac{1}{\sqrt{r^2+(1+\beta) z^2}}.
\end{equation}
This Hamiltonian has two degrees of freedom, where $V_{\varpi}$ can be seen as the potential energy of the reduced Hamiltonian.
The corresponding reduced Hamiltonian system is
\begin{equation}
\label{eq: RedHamSys} \dot{\zeta}(t) = X_{H_{\varpi}}(\zeta)= J \nabla H_{\varpi}(\zeta),\ \, \zeta=(p_r, p_z, r, z)\in \rr^2\times\rr^+\times\rr,
\end{equation}
where $X_{H_{\varpi}}$ is the associated Hamiltonian vector field, $J=\begin{pmatrix}
		0 & -I_2\\
		I_2 & 0
	\end{pmatrix}$ is the standard symplectic matrix.

In this paper, for fixed $h$ and $\varpi$, we are interested in periodic orbits on the energy surface
$\mathfrak{M}(h,\varpi,\beta)=\{(p_r,p_z,r,z)\,|\,H_{\varpi}=h\}$.
Moreover, each $T$-periodic orbit $\zeta(t)=(p_r,p_z,r,z)(t)$ of \eqref{eq: RedHamSys} corresponds to a family of relative periodic orbits of the spatial anisotropic Kepler problem. We say an orbit $q(t)=(x,y,z)(t)=(r\cos\theta,r\sin\theta,z)(t)$ in $\mathbb R^3$ is a relative periodic orbit, if there exists a $T>0$, such that
	$$
	q(t+T)=gq(t),\quad g=\left(\begin{array}{ccc}
		\cos\theta_* & -\sin\theta_* & 0\\
		\sin\theta_* &  \cos\theta_* & 0\\
		0  &   0   & 1
	\end{array}\right),
	$$
	where $\theta_*=\theta(T)-\theta(0)=\int^T_0\frac{\varpi}{r^2(t)}dt$ due to the first integral $r^2\dot{\theta}\equiv\varpi$. We refer to \cite{M06} for more general relative periodic solutions, where $g\in \mathbb O(3)$ is any orthogonal matrix.

Now, for fixed $h$ and $\varpi$, Proposition \ref{topo of energy surfaces} shows that $\mathfrak{M}(h,\varpi,\beta)$ is a compact manifold homeomorphic to $\mathbb{S}^{3}$
when $-1<2h\varpi^2<0$. Since \eqref{eq;reducedHam} is a mechanical Hamiltonian, from \cite[Theorem 4.8]{HZ94}, we know there exists a contact form $\lambda$ on $\mathfrak{M}(h,\varpi,\beta)$, which is a primitive of $\omega_0=dp_r\wedge dr+dp_z\wedge dz$, then $(\mathfrak{M}(h,\varpi,\beta),\lambda)$ forms a contact manifold. Obviously, there exists a special orbit, namely the planar Kepler orbit $\zeta_{k}(t)=(p_{r,k},0,r_k,0)(t)$ lying along the $r$-axis, which satisfies the equation
$\ddot{r}_{k}=\frac{\varpi^2}{r_k^3}-\frac{1}{r_k^2}.
$
Let $r_k(0)=\frac{\varpi^2}{1+\fe}$, $\dot{r}_k(0)=0$, then we can solve it explicitly,
$r_k(t)=\frac{\varpi^2}{1+\fe\cos\theta_{k}(t)},
$ where $\theta_{k}(t)$ satisfies $r_k^2(t)\dot{\theta_{k}}=\varpi$ and $\theta_{k}(0)=0$. Based on the symmetry, we can find disk-like global surfaces of section whose boundary is exactly the planar Kepler orbit $\zeta_{k}$. As with the treatment of the isosceles three-body problem in Hu, Liu, Ou, Qiao and Salom\~ao \cite{HLOQS26}, through a meticulous analysis of rotation number properties combined with the contact volume formula from \cite{CHHL23}, we derive our main theorem
\begin{thm}\label{thm; exist infinitely many}
	For $\beta>-1$ and $-1<2h\varpi^2<0$, the energy surface $\mathfrak{M}(h,\varpi,\beta)$ is a compact and regular contact manifold homeomorphic to a three-sphere $\mathbb{S}^{3}$ which has disk-like global surfaces of section of the Hamiltonian flow $\varphi_t$ corresponding to the reduced Hamiltonian $H_{\varpi}$ and the binding is the planar Kepler orbit $\zeta_{k}$.
	Moreover, there exist infinitely many periodic orbits on each fixed compact energy surface $\mathfrak{M}(h,\varpi,\beta)$ with $\beta\in(-1,0]$.
	
\end{thm}
\begin{rem}
	Combining the Birkhoff shooting method in \cite{HLOYS23} and the criteria in \cite{HSW23}, there exists another brake and $z$-symmetric periodic orbit $\zeta_z\subset \mathfrak M$ such that the Hopf link $\zeta_z\cup \zeta_k\subset \mathfrak M$ forms the binding of an open book decomposition whose pages are annulus-like global surface of section adapted to Hamiltonian flow. Moreover, the Seifert rotation number $ \rho(\zeta_z)\geq 0$.
\end{rem}


This paper is organized as follows: In Section $2$, we consider the topology of the energy
surfaces $\mathfrak{M}(h,\varpi,\beta)$ and estimate their contact volume. Moreover, we analyze the global surfaces of section of $\mathfrak{M}(h,\varpi,\beta)$ and provide a criterion for the existence of infinitely many periodic orbits. In Section $3$, we study the properties of degenerate curves, which can be used to estimate the mean index. Then we prove our main Theorem \ref{thm; exist infinitely many}. In Appendix, we give a brief review of the Maslov-type index for symplectic paths, the relative Morse index and spectral flow.
\section{Analysis of Energy Surfaces}
\subsection{The topology and volume of the energy surface}
In this section, for the reduced Hamiltonian system $\eqref{eq: RedHamSys}$, we consider the topology of the energy surfaces $\mathfrak{M}(h,\varpi,\beta)$ and
we estimate their volume when they are compact.
\begin{prop}\label{topo of energy surfaces}
	For $\beta>-1$, the energy surface $\mathfrak{M}(h,\varpi,\beta)$ satisfies
	\begin{enumerate}
		\item[(i)] If $-1<2h\varpi^2<0$, then $\mathfrak{M}(h,\varpi,\beta)$ is a compact manifold homeomorphic to $\mathbb{S}^{3}$.
		\item[(ii)] If $2h\varpi^2=-1$, then $\mathfrak{M}(h,\varpi,\beta)=\{(0,0,\varpi^2,0)\}$.
		\item[(iii)] If $2h\varpi^2<-1$, then $\mathfrak{M}(h,\varpi,\beta)=\varnothing$.
		\item[(iv)] If $2h\varpi^2\geq0$, then $\mathfrak{M}(h,\varpi,\beta)$ is unbounded.
	\end{enumerate}
\end{prop}
\begin{proof}
Direct computation shows that $\nabla H_{\varpi}=\left(p_{r},p_{z},-\frac{\varpi^{2}}{r^{3}}+\frac{r}{(r^{2}+(1+\beta)z^{2})^{3\slash2}},\frac{(1+\beta)z}{(r^{2}+(1+\beta)z^{2})^{3\slash2}}\right)$,
so $(0,0,\varpi^2,0)$ is the unique critical point of $H_{\varpi}$. Meanwhile, the Hessian matrix at $(0,0,\varpi^2,0)$ is
$D^{2}H_{\varpi}(0,0,\varpi^2,0)=\text{diag}(1,1,1/\varpi^6, (1+\beta)/\varpi^6)$,
which is positive definite. Hence, the unique critical
point of $H_{\varpi}$ is $(0,0,\varpi^2,0)$, which is a non-degenerate global minimum and $H_{\varpi}(0,0,\varpi^2,0)=-1/2\varpi^2$. Then $(ii), (iii)$ are proved. Since $H_\varpi$ is increasing with respect to $z^2$, when $h\geq\liminf_{z\rightarrow\infty}H_\varpi=0$,  $\mathfrak{M}(h,\varpi,\beta)$ is unbounded in $z$ direction. This proves $(iv)$.
For $(i)$, where $-1<2h\varpi^2<0$, by above, $\mathfrak{M}(h,\varpi,\beta)$ is bounded in $z$ direction. Direct computation yields
\begin{eqnarray}\label{esti r,z}
	r\in[r_{\min}, r_{\max}]=[\frac{1-\mathfrak{e}}{-2h}, \frac{1+\mathfrak{e}}{-2h}],\ \ |z|=\frac{1}{\sqrt{1+\beta}}
	\big{(}\frac{1}{(\varpi^2/2r^2-h)^2}-r^2\big{)}^{-1/2},\nonumber
\end{eqnarray}
where $\mathfrak{e}=\sqrt{1+2\varpi^2h}.$ This shows $\mathfrak{M}(h,\varpi,\beta)$ is also bounded in $r$ direction. So it is compact. By Morse theory, as $h$ increases from
below $-1/2\varpi^2$ to slightly above $-1/2\varpi^2$, $\mathfrak{M}(h,\varpi,\beta)$
changes from the empty set to a sphere-like regular hypersurface. Since $\mathfrak{M}(h,\varpi,\beta)$ is
bounded for every $h$ in that interval, it is a sphere-like hypersurface.
\end{proof}
Next, we focus on the case where the energy surface is homeomorphic to $\mathbb{S}^{3}$. For convenience, we denote $\mathfrak{M}(h,\varpi,\beta)$ by $\mathfrak{M}_{\beta}$ when there is no confusion. Since $H_\varpi$ is a mechanical Hamiltonian,  from \cite[Theorem 4.8]{HZ94}, the hypersurfaces $\mathfrak{M}_{\beta}$ have contact type. So there exists a contact form $\lambda$ on $\mathfrak{M}_{\beta}$, which is a primitive of $\omega_0=dp_r\wedge dr+dp_z\wedge dz$ and satisfies $\lambda\wedge d \lambda\neq0$. Then $(\mathfrak{M}_{\beta},\lambda)$ is a contact manifold. The Reeb vector field $R$ of $\lambda$, determined by $\mathrm{d}\lambda(R,\cdot)\equiv 0$ and $\lambda(R)=1$, is parallel to $X_{H_\varpi}$. Therefore, the Reeb flow on $(\mathfrak{M}_{\beta},\lambda)$ is dynamically equivalent to the Hamiltonian flow on $\mathfrak{M}_{\beta}$ when $\mathfrak{M}_{\beta}$ is treated as an energy surface. Hence, we alternatively study the Reeb vector field $R$. Following the definitions in \cite{CHHL23}, we define the \textbf{contact volume} of $\mathfrak{M}_{\beta}$ and \textbf{action} of the periodic orbit $\zeta(t)$ as follows
\begin{equation}
	\mathrm{vol}(\mathfrak{M}_{\beta},\lambda\wedge d\lambda)=\int_{\mathfrak{M}_{\beta}}\lambda\wedge\mathrm{d}\lambda,\ \ A(\zeta(t))=\int_{\zeta(t)}\lambda.\nonumber
\end{equation}

In order to obtain the volume of $\M_{\beta}$, we need the following useful lemma. It was proved by Hu, Qiao and Yu in \cite{HQY2025}. There they
studied the relative periodic orbits of the spatial Kepler problem (i.e $\beta=0$) under a perturbation that satisfies both rotational symmetry and reflection symmetry with respect to a plane perpendicular to the rotational axis.
\begin{lem}(\cite{HQY2025})\label{vom 0}
For $\beta=0$ and $-1<2h\varpi^2<0$, the volume satisfies
	$$
	\mathrm{vol}(\mathfrak{M}_{0},\lambda\wedge d\lambda)=A_k^2,
	$$
where $A_k=A(\zeta_{k}(t))=2\pi\varpi\left(\frac{1}{\sqrt{1-\fe^{2}}}-1\right)$ denotes the action of the planar Kepler orbit $\zeta_{k}$.
\end{lem}
Based on this lemma and a scaling transformation, we further obtain the following result for $\beta>-1$.
\begin{prop}\label{prop: lower bound of vol}
For $\beta>-1$ and $-1<2h\varpi^2<0$, the volume satisfies
	$$
	\mathrm{vol}(\mathfrak{M}_{\beta},\lambda\wedge d\lambda)=\frac{A_k^2}{\sqrt{1+\beta}}.
	$$
\end{prop}
\begin{proof}
		By means of the scaling transformation $\Phi_\beta:(p_r,p_z,r,z)\rightarrow(p_r,p_z,r,\sqrt{1+\beta}z)$, we have $\mathfrak{M}_\beta=\Phi_\beta(\mathfrak{M}_0)$. Thus,
		$$
		\begin{aligned}
			\text{vol}(\mathfrak{M}_{\beta},\lambda\wedge d\lambda)&=\int_{\mathfrak{M}_{\beta}}\lambda\wedge\mathrm{d}\lambda=\int_{\Phi_\beta(\mathfrak{M}_0)}\lambda\wedge\mathrm{d}\lambda=\int_{\mathfrak{M}_0}\Phi_\beta^*(\lambda\wedge\mathrm{d}\lambda)=\frac{1}{\sqrt{1+\beta}}\int_{\mathfrak{M}_0}\lambda\wedge\mathrm{d}\lambda\\
			&=\frac{\text{vol}(\mathfrak{M}_{0},\lambda\wedge d\lambda)}{\sqrt{1+\beta}}=\frac{A_{k}^2}{\sqrt{1+\beta}},
		\end{aligned}
		$$
		where the last equality follows from Lemma \ref{vom 0}.
\end{proof}
\subsection{The global surfaces of section}
In this section, we study the global surfaces of section, which are useful in the study of periodic orbits.
\begin{defi} \label{defi: GlobalSurface} \cite[Chapter 9]{FO18} Let $X$ be a non-vanishing vector field on  $\mathbb{S}^3$ and $\varphi^X_t$ denote its corresponding flow.
A disk-like global surface of section is an embedded closed $2$-dimensional disk $\mathbb{D }\subset \mathbb{S}^3$ satisfying
\begin{enumerate}
\item[(i.)] $X$ is tangent to $\partial \mathbb{D}$, the boundary of $\mathbb{D}$;
\item[(ii.)] $X$ is transverse to $\mathbb{D}^{\circ}$ the interior of $\mathbb{D}$;
\item[(iii.)] for any $x \in \mathbb{S}^3 \setminus \partial \mathbb{D}$, there exist $t^+>0$ and $t^- <0$, such that both $\varphi_{t^+}^X(x)$ and $\varphi_{t^-}^X(x)$ are contained in the interior of $\mathbb{D}$.
\end{enumerate}
Clearly, $\partial \mathbb{D}$ is a periodic orbit of $X$ and is called the binding orbit.
\end{defi}
\begin{thm} \label{thm: GlobalSurface}
For $\beta>-1$ and $-1<2h\varpi^2<0$, the sets $\Sigma_{\pm}$ (defined below) are two disk-like global surfaces of section
for the Hamiltonian flow $\varphi_{t}$ corresponding to the Hamiltonian vector field $X_{H_{\varpi}}$,
$$
\begin{aligned}
&\Sigma_{+}=\{(p_{r},p_{z},r,z)\in\mathfrak{M}_{\beta}: z=0, p_{z}\geq0\},\\
&\Sigma_{-}=\{(p_{r},p_{z},r,z)\in\mathfrak{M}_{\beta}: z=0, p_{z}\leq0\},
\end{aligned}
$$
the binding orbit is precisely the planar Kepler orbit $\zeta_{k}(t)$.
\end{thm}
\begin{proof}
		We only consider $\Sigma_{+}$, the proof for $\Sigma_{-}$ is similar.
		First, by the above definition, one can see that the boundary $\partial \Sigma_{+}=\{(p_{r},p_{z},r,z)\in\mathfrak{M}_{\beta}: z=0, p_{z}=0\}$ coincides with the planar Kepler
		orbits $\zeta_{k}$, which is tangent to $X_{H_{\varpi}}$.
		Second, for any $\zeta=(p_{r},p_{z},r,z)\in \Sigma_{+}^{\circ}$, the interior of $\Sigma_{+}$, we have $z=0$. The tangent space $T_{\zeta}(\Sigma_{+}^{\circ})$ is $0$ along the direction $\partial_z$, while the component of the vector field $X_{H_{\varpi}}=J\nabla H_{\varpi}(\zeta)$ along the direction $\partial_z$ is $p_{z}>0$, this ensures that $X_{H_{\varpi}}$ is
		transverse to $\Sigma_{+}^{\circ}$. Last, for any orbit $\zeta(t)=(p_{r}(t),p_{z}(t),r(t),z(t))$ starting at $\zeta(0)\in\Sigma_{+}^{\circ}$, from equation (\ref{eq: RedHamSys}), one can see that there exists a constant $c>0$ such that
		$$\ddot{z}(t)=-\frac{(1+\beta)z(t)}{(r(t)^{2}+(1+\beta)z(t)^{2})^{3\slash2}}\text{ and } \frac{(1+\beta)}{(r(t)^{2}+(1+\beta)z(t)^{2})^{3\slash2}}\geq c.
		$$ The reasion is $r(t)$ and $z(t)$ are bounded in $\M_{\beta}$ by Proposition \ref{topo of energy surfaces}. So we have 
		$\ddot{z}(t)\leq-cz(t)$ for $z(t)\geq0$ and $\ddot{z}(t)\geq-cz(t)$ for $z(t)\leq0$. By Sturm Comparison Theorem, this implies that there exist $t_{+}>t_{1}>0$ such that $p_{z}(t_{1})<0, z(t_{1})=0$ and $p_{z}(t_{+})>0, z(t_{+})=0$, that is $\zeta(t_{1})\in \Sigma_{-}^{\circ}$ and $\zeta(t_{+})\in \Sigma_{+}^{\circ}$. This completes the proof.
\end{proof}
\subsection{Applications to periodic orbits}
In this section, we give a criterion for the existence of infinitely many periodic orbits on $\mathfrak{M}(h,\varpi,\beta)$. This is based on the volume formula in \cite{CHHL23} and Franks' theorem in \cite{F96}. In general, a Reeb orbit is a periodic orbit of Reeb vector field. Assume $\zeta$ is a Reeb orbit and $\gamma$ its monodromy matrix. We say $\zeta$ is hyperbolic if all the eigenvalues of $\gamma$ are real but not equal to $\pm1$. We say $\zeta$ is elliptic if all the eigenvalues of $\gamma$ are on the unit circle $\mathbb{U}$ and are semi-simple. Then we have
\begin{thm}\label{thm:The condition of only two}
	\cite[Theorem 1.2 and 1.5]{CHHL23} Let $S$ be a three-dimensional sphere, and $\lambda$ be a contact form on $S$ with exactly two Reeb  orbits, $\zeta_{1}$ and $\zeta_{2}$. Let $A_{i}$ and $\rho_{i}$ be the action and  Seifert rotation number of $\zeta_{i},\ i=1,2$ and $\emph{vol}(S,\lambda\wedge d\lambda)$ be the contact volume of $S$. Then
	\begin{equation}\label{eq:Kepler equal}
		\emph{vol}(S,\lambda\wedge d\lambda)=\frac{A_{1}^{2}}{\rho_{1}}=\frac{A_{2}^{2}}{\rho_{2}},\nonumber
	\end{equation}
	and both Reeb orbits are irrationally elliptic.
\end{thm}
\begin{rem}\label{sef rotation}
From Proposition \ref{topo of energy surfaces}, when $\beta>-1, -1<2h\varpi^2<0$, the energy surface $\mathfrak{M}(h,\varpi,\beta)$ is a compact and regular manifold homeomorphic to $\mathbb{S}^{3}$ with binding orbits $\zeta_{k}$.
From (5-2) in \cite{CHHL23}, the Seifert rotation number $\rho_{kep}$ of the binding orbit is related to its mean index $\hat{i}(\gamma_{kep})$
(See the Definition \ref{def:Maslov-type index} of the mean index in the Appendix),
\begin{equation}\label{Sef rotaion and mean index}
\rho_{kep}=\frac{\hat{i}(\gamma_{kep})}{2}-1,
\end{equation}
where $\gamma_{kep}$ is the fundamental solution of the linearized equation of Hamiltonian system (\ref{eq: RedHamSys}) along $\zeta_{k}$.
\end{rem}
Moreover, based on the global surface of section, one can see that the Poincar\'{e} map $g: \Sigma_{+}\rightarrow \Sigma_{+}$ contains essentially all the relevant information on the Hamiltonian flow $\varphi_{t}(x)$ on the three-dimensional
energy surface $\mathfrak{M}(h,\varpi,\beta)$. Instead of the continuous flow on the three-dimensional manifold, we
can study the Poincar\'{e} map. Periodic orbits of $\varphi_{t}(x)$ different from the binding orbit $\zeta_{k}$ correspond to interior periodic points of
the Poincar\'{e} map on $\Sigma_{+}^{\circ}$.

As mentioned in \cite{HWZ98}, the Poincar\'{e} map $g: \Sigma_{+}^{\circ}\rightarrow \Sigma_{+}^{\circ}$ is
conjugated to an area preserving diffeomorphism $\psi: \mathbb{D}^{\circ}\rightarrow \mathbb{D}^{\circ}$ of the open disk $\mathbb{D}^{\circ}=\{z\in\mathbb{C}| |z|<1\}$.
By Brouwer's translation theorem \cite{Fr92}, $\psi$ must have a fixed point $z^{*}\in\mathbb{D}^{\circ}$. Then we can restrict the map
$\psi$ to the open annulus $\mathbb{D}^{\circ}\setminus\{z^*\}$ and apply the following theorem due to Franks to obtain Corollary \ref{infini peri orbit}.
\begin{thm}(\cite{F96})\label{Franks thm}
An area preserving homeomorphism of the open annulus which has at least one periodic point must in fact have infinitely many interior
periodic points.
\end{thm}
\begin{cor}\label{infini peri orbit}
For $\beta>-1, -1<2h\varpi^2<0$, if $\rho_{kep}\neq A_{k}^2/\emph{vol}(\mathfrak{M}_{\beta},\lambda\wedge d\lambda)=\sqrt{1+\beta}$ or $\rho_{kep}\in\mathbb{Q}$,
then the spatial anisotropic Kepler problem has infinitely many relative periodic orbits on $\mathfrak{M}(h,\varpi,\beta)$.
\end{cor}

\section{Proof of the Main Theorem}

\subsection{Seifert rotation number of planar Kepler orbit $\zeta_{k}$}
In this section, we will study the Seifert rotation number of $\zeta_{k}$ and the process is similar to the Euler orbit in the spatial isosceles three body problem \cite{HLOYS23}. Consider the fundamental solution $\gamma_{kep}$ along the binding orbit $\zeta_k$, we have
$$\dot{\gamma}_{kep}(t)=J\mathcal{B}(t)\gamma_{kep}(t),\ \ \gamma_{kep}(0)=I.
$$	
where $\mathcal{B}(t)=diag(1,1,(3\varpi^2-2r_k(t))r_k(t)^{-4},(1+\beta)r_k(t)^{-3})$.
Then it can be decomposed into two subsystems,
\begin{equation}\label{linear equ1}
\dot{\gamma}_{1}(t)=J_2\left(\begin{array}{cc}
	1 & 0\\
	0 & (3\varpi^2-2r_k(t))r_k(t)^{-4}
\end{array}\right)\gamma_{1}(t),
\end{equation}
and
\begin{equation}\label{ess equ}
\dot{\gamma}_{2}(t)=J_2\left(\begin{array}{cc}
	1 & 0\\
	0 & (1+\beta)r_k(t)^{-3}
\end{array}\right)\gamma_{2}(t).
\end{equation}
One can check that $\dot{\zeta}_k(t) = (\ddot{r}_k(t), 0, \dot{r}_k(t), 0)$  satisfies $\ddot{\zeta}_k(t)=J\mathcal{B}(t)\dot{\zeta}_k(t)$. Then $\xi_1(t)=c\cdot(\ddot{r}_k(t),\dot{r}_k(t))$ is a periodic solution of \eqref{linear equ1}.  Choosing $c$ properly, we have $\xi_1(0)=(1,0)^T$ and solution $\xi_1$ becomes the first column of $\gamma_{1}$.  Let $\xi_2$ be the second solution of  \eqref{linear equ1} which satisfies $\xi_2(0)=(0,1)^T$, then $\gamma_{1}=(\xi_1, \xi_2)$. Direct computation shows that
\begin{eqnarray*}
\begin{aligned}
\sum_{0<\hat{t}\leq mT} \operatorname{dim} \gamma_{1}\left(\hat{t}\right) \Lambda_{D} \cap \Lambda_{D}=\#\{\hat{t}:\dot{r}_{k}(\hat{t})=0,t\in(0,mT]\}
=\#\{\hat{t}:\sin\theta(\hat{t})=0,\hat{t}\in(0,mT]\}
=2m-1,
\end{aligned}
\end{eqnarray*}
where $\Lambda_D=\mathbb{R}\oplus\{0\}$ is the Lagrangian subspace of $\rr^2$. From (\ref{mean index2}) in Appendix, we have
\begin{eqnarray}\label{mean index1}
\hat{i}(\gamma_{1})=\lim_{k\to +\infty}\frac{1}{m}\sum_{0<\hat{t}\leq mT} \operatorname{dim} \gamma_{1}\left(\hat{t}\right) \Lambda_{D} \cap \Lambda_{D}=2.\nonumber
\end{eqnarray}
For $\hat{i}(\gamma_{2})$, based on $\dot{\theta}(t)=\varpi/r^2_{k}(t)$, we can
change the time variable $t$ to $\theta$ and denote $'=d/d\theta$, then
$$\gamma_{2}'(t(\theta))=J_2\left(\begin{array}{cc}
	r_k(\theta)^2\varpi^{-1} & 0\\
	0 & (1+\beta)r_k(\theta)^{-1}\varpi^{-1}
\end{array}\right)\gamma_{2}(t(\theta)),
$$To further simplify the above linear system, we use a time-dependent linear symplectic transformation and let
$$\gamma_{\beta,\fe}(\theta)=\mathcal{R}(\theta)\gamma_{2}(t(\theta)),\ \
\text{where}\ \ \mathcal{R}(\theta)=\left(\begin{array}{cc}
\frac{r_k}{\sqrt{\varpi}} & -\frac{\fe\sin\theta}{\sqrt{\varpi}^3}\\
0 &\frac{\sqrt{\varpi}}{r_k}
\end{array}\right).
$$
Direct computation shows that
\begin{equation}\label{eq;linearized equation}
	\gamma'_{\beta,\fe}(\theta)=J_2\left(\begin{array}{cc}
		1 & 0\\
		0 & 1+\beta(1+\fe\cos\theta)^{-1}
	\end{array}\right)\gamma_{\beta,\fe}(\theta).
\end{equation}
From the symplectic invariance of the mean index, we have
$
\hat{i}(\gamma_{2})=\hat{i}(\gamma_{\beta,\fe}),
$
hence
$$
\hat{i}(\gamma_{kep})=\hat{i}(\gamma_{1})+\hat{i}(\gamma_{2})=2+\hat{i}(\gamma_{\beta,\fe}).
$$
Combined with (\ref{Sef rotaion and mean index}) in Remark \ref{sef rotation}, we obtain
\begin{equation*}\label{Sef rotaion and mean index2}
\rho_{kep}=\frac{\hat{i}(\gamma_{\beta,\fe})}{2},\ \ \text{for}\ \ \beta>-1.
\end{equation*}
In the following section, we will analyze the properties of the mean index $\hat{i}(\gamma_{\beta,\fe})$.
\subsection{Fundamental properties of degenerate curves}
Consider the following family of operators on $L^2([0,2\pi], \mathbb{C})$ with domain $\bar{D}(\omega,2\pi)$ for $\omega\in\mathbb{U}$.
$$
\mathcal{A}(\beta,\fe)=-\frac{d^2}{d\theta^2}-1-\frac{\beta}{1+\fe\cos \theta}\quad \text{where}\ \beta>-1,\fe\in(0,1)
$$
where
$$
\bar{D}(\omega,2\pi)=\{W^{2,2}([0,\tau], \mathbb{C}^n)|y(2\pi)=\omega y(0), \dot{y}(2\pi)=\omega \dot{y}(0)\}.
$$
From index relation (\ref{index equ}) in Appendix, we have
\begin{equation}\label{index equ1}
\phi_{\omega}(\mathcal{A}(\beta,\mathfrak{e}))=i_{\omega}(\gamma_{\beta,\mathfrak{e}}), \ \ v_{\omega}(\mathcal{A}(\beta,\mathfrak{e}))=\nu_{\omega}(\gamma_{\beta,\mathfrak{e}}),
\end{equation}
where $\phi_\omega(\mathcal{A})$ denotes the $\omega$-Morse index of $\mathcal{A}$.

In \cite{Z19}, Zhou first used the Maslov-type index theory to study the properties of
$\mathcal{A}(\beta,\fe)$ for $\beta\in(-1,0]$. This relates to the stability of $\gamma_{\beta,\fe}(2\pi)$, which corresponds to the spatial part of the Robe's restricted three-body problem. Further, the
study of $\mathcal{A}(\beta,\fe)$ for any $\beta>-1$ is given by Hu, Ou and Tang \cite{HOT23}. Following \cite{HOT23}, we say that $\mathcal{A}(\beta,\mathfrak{e})$ is $\omega$-degenerate, if $v_{\omega}(\mathcal{A}(\beta,\mathfrak{e}))\neq 0$. For fixed $\omega$-boundary condition and $\fe\in[0,1)$, let $\beta_i(\fe,\omega), i\in\mathbb{Z}^{+}$ be the $i$-th eigenvalue (counting multiplicities) of $(1+\fe\cos\theta)(-d^2/d\theta^2-1),$ that is $\mathcal{A}(\beta_i(\mathfrak{e}, \omega),\mathfrak{e})$ is
$\omega$-degenerate. Then we have
$
\beta_i(\fe,\omega)=\beta_i(\fe,\bar{\omega})
$
and
$\beta_i(\fe,\omega)\leq\beta_j(\fe,\omega)$ for $i<j$.

We define $n$-th $\omega$-degenerate curve by $\Gamma_n(\omega)=(\beta_n(\fe,\omega),\fe)$. The following numerical figure is taken from \cite{Z19}, we can see the distribution of
these degenerate curves $\Gamma_1(-1)$ and $\Gamma_2(-1)$, which form the boundary of the stable region.
\begin{figure}[H]
\centering
\includegraphics[height=
0.5\textwidth,width=0.83\textwidth]{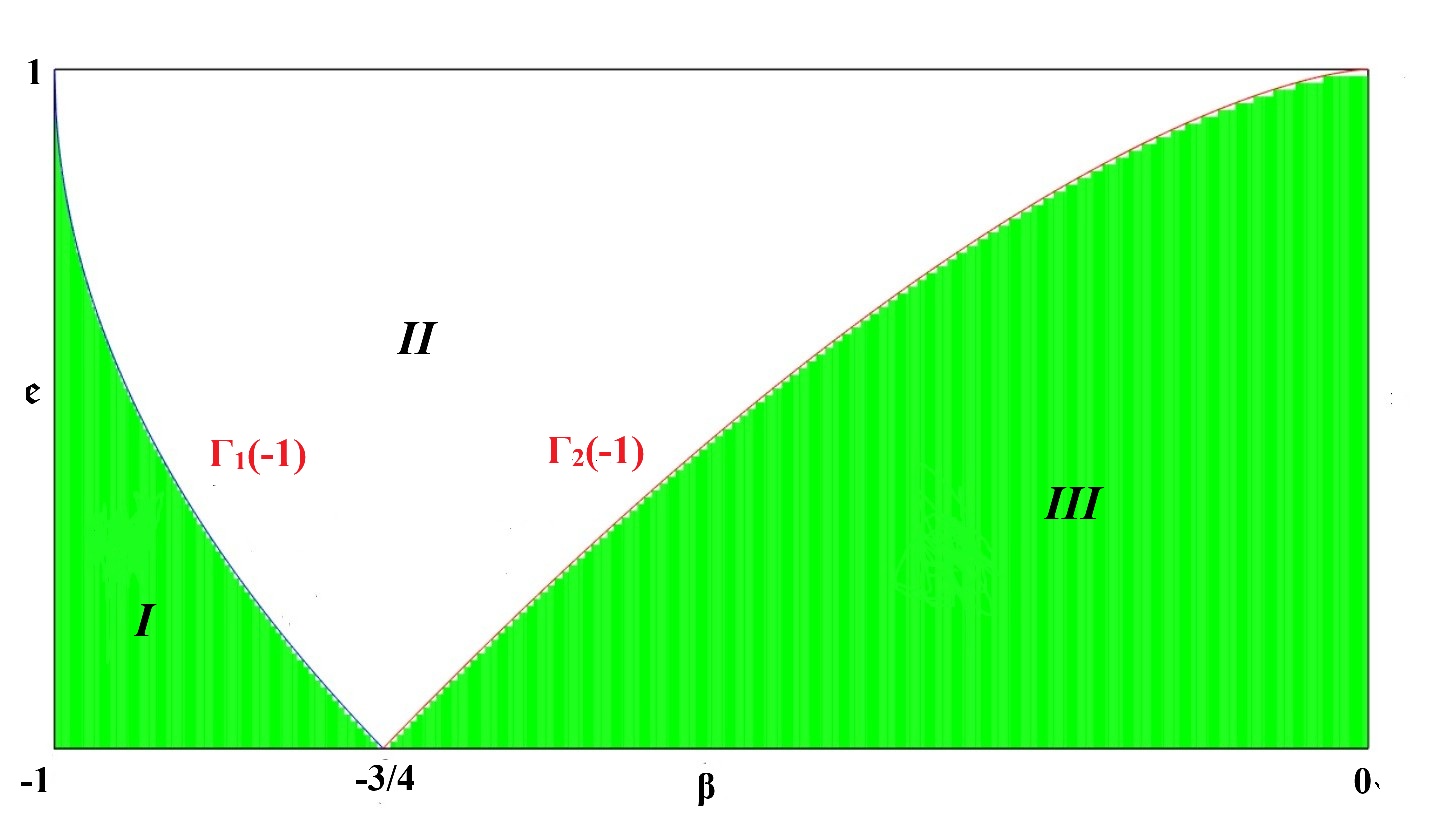}
\caption{$I, III$ are stable region, while $II$ is the hyperbolic region.}
\label{picture of Birfurcation}
\end{figure}
In this paper, we focus on the case $\beta\in(-1,0]$. We list the properties of $\mathcal{A}(\beta,\fe)$ in the following Theorem \ref{prop of degcurve}, which will be used in our arguments.
\begin{thm}(\cite{HOT23,Z19})\label{prop of degcurve}
Consider $\beta\in(-1,0]$, then for any fixed $\mathfrak{e}\in[0,1), \omega\in\mathbb{U}$, $i_{\omega}(\gamma_{-1,\mathfrak{e}})=0$ and $\beta\mapsto i_{\omega}(\gamma_{\beta,\mathfrak{e}})$ is non-decreasing. More precisely, $\lim_{\epsilon\rightarrow0^+}i_{\omega}(\gamma_{\beta+\epsilon,\mathfrak{e}})=i_{\omega}(\gamma_{\beta,\mathfrak{e}})+\nu_{\omega}(\gamma_{\beta,\mathfrak{e}})$ and
$i_{\omega}(\gamma_{\beta,\mathfrak{e}})$ is strictly increasing only
when $\beta$ crosses $\beta_{j}(\mathfrak{e},\omega),j=1,2$. Furthermore, the curves $\Gamma_{j}(\omega)$ possess the following properties:

\vskip 0.15 cm
i) For $\mathfrak{e}\in(0,1), \omega\in\mathbb{U}\setminus\{1, -1\}$, we have $-1=\beta_{1}(\mathfrak{e},1)<\beta_{1}(\mathfrak{e},\omega)<\beta_{1}(\mathfrak{e},-1)<\beta_{2}(\mathfrak{e},-1)<\beta_{2}(\mathfrak{e},\omega)<\beta_{2}(\mathfrak{e},1)
=\beta_{3}(\mathfrak{e},1)=0.$
For $\mathfrak{e}=0$, 
$\partial_{\mathfrak{e}}\beta_{1}(0,\omega)=\partial_{\mathfrak{e}}\beta_{2}(0,\omega)=0$ for $\omega\in\mathbb{U}\setminus\{-1\}$ and $\partial_{\mathfrak{e}}\beta_{1}(0,-1)=-\partial_{\mathfrak{e}}\beta_{2}(0,-1)=-3/8$.

ii) For fixed $\omega\in\mathbb{U}$, when $\mathfrak{e}\rightarrow 1^-$, $\beta_{1}(\mathfrak{e},\omega)\rightarrow-1$ and $\beta_{2}(\mathfrak{e},\omega)\rightarrow 0.$

iii) For any $\mathfrak{e}\in[0,1)$, $\gamma_{\beta_{1}(\mathfrak{e},1),\mathfrak{e}}(2\pi)\approx N(1,-1),\, \gamma_{\beta_{2}(\mathfrak{e},1),\mathfrak{e}}(2\pi)=\gamma_{\beta_{3}(\mathfrak{e},1),\mathfrak{e}}(2\pi)\approx I_{2},$ and
$i_{1}(\gamma_{\beta_{1}(\mathfrak{e},1),\mathfrak{e}})=0,\, i_{1}(\gamma_{\beta_{2}(\mathfrak{e},1),\mathfrak{e}})=i_{1}(\gamma_{\beta_{3}(\mathfrak{e},1),\mathfrak{e}})=1.$

iv) $\gamma_{\beta,\mathfrak{e}}(2\pi)\approx R(\vartheta)$ for some $\vartheta\in(0,\pi)$ when $\beta_{1}(\mathfrak{e},1)<\beta<\beta_{1}(\mathfrak{e},-1)$ and $\vartheta\in(\pi,2\pi)$ when  $\beta_{2}(\mathfrak{e},-1)<\beta<\beta_{2}(\mathfrak{e},1)$. In this case, $\gamma_{\beta,\mathfrak{e}}(2\pi)$ is strongly linearly stable and $i_{1}(\gamma_{\beta,\mathfrak{e}})=1.$

v) $\gamma_{\beta,\mathfrak{e}}(2\pi)\approx D(\lambda)$ for some $0>\lambda\neq -1$ when $\beta_{1}(\mathfrak{e},-1)<\beta<\beta_{2}(\mathfrak{e},-1)$. In this case, $\gamma_{\beta,\mathfrak{e}}(2\pi)$ is hyperbolic and $i_{1}(\gamma_{\beta,\mathfrak{e}})=1$.

Here $\approx$ means symplectic similarity and $R(\vartheta), D(\lambda)$ and $N_{1}(1,a), a=\pm1$ is given in Appendix.
\end{thm}
Based on Theorem \ref{prop of degcurve}, we obtain further properties of the mean index $\hat{i}(\gamma_{\beta,\fe})$.
\begin{cor}\label{mean index monotonicity}
For $\fe=0$, the mean index satisfies $\hat{i}(\gamma_{\beta,0})=2\sqrt{1+\beta}$. For any fixed $\fe\in[0,1)$, the mean index is strictly increasing in $\beta\in(-1,0]\setminus[\beta_{1}(\fe,-1), \beta_{2}(\fe,-1)]$,
while $\hat{i}(\gamma_{\beta,\fe})=1$ for $\beta\in[\beta_{1}(\fe,-1), \beta_{2}(\fe,-1)]$, and $\hat{i}(\gamma_{\beta,\fe})=2$ for $\beta=\beta_{2}(\fe,1)$.
\end{cor}
\begin{proof}
It follows from the facts that $\gamma_{\beta,\mathfrak e}(2\pi)\approx R(\theta)$ for some $\theta\in (0,\pi)\cup(\pi,2\pi)$, where $\theta=\theta(\beta)$ is strictly increasing with $\beta$ on $(-1,\beta_1(\fe,-1))\cup(\beta_2(\fe,-1),0)$, and that the eigenvalues of $\gamma_{\beta,\mathfrak e}(2\pi)$ are invariant along the $\omega$-degenerate curves. The mean index can be computed using Theorem \ref{prop of degcurve} and mean index formula \eqref{mean index 2}.
\end{proof}
\subsection{Degenerate curves estimation}
In this section, we further establish the following propositions to estimate the location of the degenerate curves in the region $(\beta,\fe)\in(-1,0]\times(0,1)$. They will be useful for estimating the mean index.
\begin{prop}\label{lem;beta -1 to -3/4}
    The Morse index
    $$\phi_\omega(\mathcal{A}(\beta_{1}(0,\omega),\fe))\geq1\ \ \text{and}\ \ \beta_{1}(\fe,\omega)<
    \beta_{1}(0,\omega).
    $$
\end{prop}
\begin{proof}
For $\beta\in(-1,-3/4]$, let $\omega=e^{2\pi i\nu},\,\nu=\sqrt{1+\beta}\in(0,1/2]$, by direct computation, we have $\beta_{1}(0,\omega)=\beta$. Now consider the Fourier basis
	$\left\lbrace u_n(\theta)=e^{i(n+\nu)\theta}|n\in\mathbb{Z}\right\rbrace
	$
and let $\mathfrak{D}=\text{span}\{u_{-1}(\theta), u_{0}(\theta), u_{1}(\theta)\}$, which is a subspace of $\bar{D}(\omega,2\pi)$.
Direct computation yields
$$(\left<\mathcal{A}(\beta,\fe)u_k,u_l\right>)_{-1\leq k,l\leq 1}:=\frac{2\pi(1-\nu^2)}{\sqrt{1-\fe^2}}\begin{pmatrix}
	\lambda_1 & s & s^2 \\
	s & \lambda_0 & s \\
	s^2 & s & \lambda_{-1}\\
\end{pmatrix},
$$ where
$\lambda_1=\frac{\sqrt{1-\fe^2}}{1-\nu^2}(\nu^2+2\nu)+1$,
$\lambda_{-1}=\frac{\sqrt{1-\fe^2}}{1-\nu^2}(\nu^2-2\nu)+1$
and $\lambda_0=-\fe s$ with
$s=\frac{-1+\sqrt{1-\fe^2}}{\fe}\in(-1,0)$.
Direct computation shows that
	\begin{equation}
		\begin{aligned}
			\det \begin{pmatrix}
	\lambda_1 & s & s^2 \\
	s & \lambda_0 & s \\
	s^2 & s & \lambda_{-1}\\
\end{pmatrix}&=-s\left|\begin{array}{ccc}
				1 & -\fe & 1\\
				\lambda_1 & s & s^2\\
				s^2 & s & \lambda_{-1}
			\end{array}\right|
			=-s\left((\lambda_1+\lambda_{-1}-2)s+(\lambda_1\lambda_{-1}-1)\fe\right)\\
			&=s\sqrt{1-\fe^2}(1-1/(1-\nu^2))(2s+2\fe-\fe\sqrt{1-\fe^2}-3\fe\sqrt{1-\fe^2}/(1-\nu^2))\\
			&=s\fe(1-\fe^2)(1-1/(1-\nu^2))(s^2-3/(1-\nu^2))<0,
		\end{aligned}\nonumber
	\end{equation}
where the last equality uses the identity $2s+2\fe-\fe\sqrt{1-\fe^2}=\fe\sqrt{1-\fe^2}s^2$.
Therefore, there exists $u(\theta)\in \mathfrak{D}$ such that $\langle\mathcal{A}(\beta,\fe)u(\theta),u(\theta)\rangle<0$, which implies $\phi_\omega(\mathcal{A}(\beta_1(0,\omega),\fe))\geq1$ for any $\fe\in(0,1)$. Moreover,
from (\ref{index equ1}), Theorem \ref{prop of degcurve} and Corollary \ref{mean index monotonicity}, we know that $\phi_\omega(\mathcal{A}(-1,\fe))=0$, $\beta\mapsto \phi_\omega(\mathcal{A}(\beta,\fe))$ is non-decreasing, and  it is strictly increasing only when $\beta$ crosses $\beta_{j}(\mathfrak{e},\omega),j=1,2$. Hence
$\phi_\omega(\mathcal{A}(\beta,\fe))=0$ for $\beta\in(-1,\beta_{1}(\fe,\omega)]$. This, combined with $\phi_\omega(\mathcal{A}(\beta_1(0,\omega),\fe))\geq1$, implies
$\beta_{1}(\fe,\omega)<\beta_{1}(0,\omega)$ for any $\fe\in(0,1)$.
\end{proof}
\begin{prop}\label{prop: the second derivitive}
	For every $\omega\in \mathbb{U}\backslash\{\pm1\}$, we have $\partial_{\fe}^2\beta_2(0,\omega)=-3\beta_2(0,\omega)\cdot \frac{1+\beta_2(0,\omega)}{3+4\beta_2(0,\omega)}>0.$
\end{prop}
\begin{proof}
	Assume $x_{\fe}$ satisfies $\mathcal{A}(\beta_2(\fe,\omega),\fe)x_{\fe}=0$, which is equivalent to $$-\ddot{x}_{\fe}-x_{\fe}-\beta_2(\fe,\omega)x_{\fe}=\fe\cos\theta(\ddot{x}_{\fe}+x_{\fe}).
	$$
	When $\fe=0$, we easily know $x_0=e^{i\sqrt{1+\beta_2(0,\omega)}\theta}$ is the base of the solution space. By Theorem \ref{prop of degcurve}-$i)$, $\partial_{\fe}\beta_{2}(0,\omega)=0$
	for  $\omega\in \mathbb{U}\backslash\{\pm1\}$. Taking the derivative of the above equation with respect to $\mathfrak e$ at $\mathfrak e=0$ and denoting $\partial_{\fe}x_{\fe}|_{\fe=0}$ by $y_0$, we have
	$$-\ddot{y}_0-y_0-\beta_2(0,\omega) y_0=\cos\theta(\ddot{x}_0+x_0) =-\frac{\beta_2(0,\omega)}{2}\left(e^{i(\sqrt{1+\beta_2(0,\omega)}+1)\theta}+e^{i(\sqrt{1+\beta_2(0,\omega)}-1)\theta}\right).
	$$
	By solving the equation above, we obtain that
	$$y_0=\mu_1e^{i(\sqrt{1+\beta_2(0,\omega)}+1)\theta}+\mu_{-1}e^{i(\sqrt{1+\beta_2(0,\omega)}-1)\theta}+\mu_0e^{i\sqrt{1+\beta_2(0,\omega)}\theta},$$
	where $2\mu_{\pm 1}=-\beta_2(0,\omega)((\sqrt{1+\beta_2(0,\omega)}\pm 1)^2-(1+\beta_2(0,\omega)))^{-1}$ and $\mu_0$ is any constant.
		
	Taking the derivative of $\langle\mathcal{A}(\beta_2(\fe,\omega),\fe)x_{\fe},x_{\fe}\rangle=0$ in $\fe$ and combining with the fact that $\mathcal{A}(\beta_2(\fe,\omega),\fe)$ is a self-adjoint operator, we have
	$$\partial_{\mathfrak{e}}\beta_2(\fe,\omega)
	=\frac{\int_{0}^{2\pi}\cos\theta (\ddot{x}_{\mathfrak{e}}+x_{\mathfrak{e}})(\overline{\ddot{x}_{\mathfrak{e}}
			+x_{\mathfrak{e}}})d\theta}{\beta_2(\fe,\omega)\int_{0}^{2\pi}|x_{\mathfrak{e}}|^2(1+\mathfrak{e}\cos\theta)^{-1}d\theta}
	=\frac{\int_{0}^{2\pi}\cos\theta(\ddot{x}_{\mathfrak{e}}+x_{\mathfrak{e}})(\overline{\ddot{x}_{\mathfrak{e}}+x_{\mathfrak{e}}})d\theta}{-\int_{0}^{2\pi}(\ddot{x}_{\mathfrak{e}}
		+x_{\mathfrak{e}})\bar{x}_{\mathfrak{e}}d\theta}.
	$$
	Then we obtain
	$$\begin{aligned}
		\partial^2_{\fe}\beta_2(0,\omega)&=\frac{2\int_{0}^{2\pi}\cos\theta\cdot \mathrm{Re}((\ddot{y}_0+y_0)(\overline{\ddot{x}_0+x_0}))d\theta}{-\int_{0}^{2\pi}(\ddot{x}_0+x_0)\bar{x}_0d\theta}\\
		&=\frac{-3(1+\beta_2(0,\omega))\beta_2(0,\omega)^2\int_0^{2\pi}2\cos^2\theta d\theta}{(3+4\beta_2(0,\omega))\cdot 2\pi \beta_2(0,\omega)} =-3\beta_2(0,\omega)\frac{1+\beta_2(0,\omega)}{3+4\beta_2(0,\omega)}.
	\end{aligned}$$
	Here we used
	$$\begin{aligned}
		\mathrm{Re}((\ddot{y}_0+y_0)(\overline{\ddot{x}_0+x_0}))&=\beta_2(0,\omega)^2\mathrm{Re}\left(\mu_1e^{i\theta}+\mu_{-1}e^{-i\theta}+\mu_0-\cos\theta\right)\\
		&=\beta_2(0,\omega)^2\left((\mu_{1}+\mu_{-1}-1)\cos\theta+\mu_0\right)\\
		&=\beta_2(0,\omega)^2\left(-\frac{3(1+\beta_2(0,\omega))}{3+4\beta_2(0,\omega)}\cos\theta+\mu_0\right).
	\end{aligned}$$
	Direct computation yields $-3/4<\beta_2(0,\omega)<0$ for $\omega\in\mathbb{U}\backslash\{\pm1\}$, so $\partial_{\fe}^2\beta_2(0,\omega)>0$.
\end{proof}

\begin{prop}\label{pro;beta -3/4 to 0}
    For $\omega\neq1$, $\beta_2(\fe,\omega)$ is increasing with respect to $\fe$.
\end{prop}
\begin{proof}
Consider the operator $\mathcal{A}(\beta,\fe)$ with domain $\bar{D}(\omega,2\pi)$. For $\beta\in(-1,0]$, define
$$\mathcal{A}_0=-\frac{d^2}{d\theta^2}-1,\ \ \mathcal{B}_{\fe}=\frac{|\beta|}{1+\fe\cos\theta},$$ so that $\mathcal{A}(\beta,\fe)=\mathcal{A}_0+\mathcal{B}_{\fe}$.
From Theorem \ref{prop of degcurve}-$i)$, we have $\partial_{\mathfrak{e}}\beta_{2}(0,-1)=3/8$, which implies that $\beta_{2}(\fe,-1)\in(-3/4,0)$ when $\fe>0$ is sufficiently small. Proposition \ref{prop: the second derivitive} implies that $\beta_{2}(\fe,\omega)\in(\beta_2(0,\omega),0)$ for $\omega\neq\pm1$ and sufficiently small $\fe>0$. From Theorem \ref{prop of degcurve}-$ii)$, we have $\lim_{\fe\rightarrow1^-}\beta_{2}(\fe,\omega)=0$ for all $\omega\in\mathbb{U}$. So we only need consider $\beta\in(\beta_{2}(0,\omega),0)$, then
both $\mathcal{B}_{0}^{-1}+\mathcal{A}^{-1}_0$ and $\mathcal{B}_{\fe}^{-1}+\mathcal{A}^{-1}_0$  are non-degenerate in $\bar{D}(\omega,2\pi)$ when $\fe$ is close to $1$.
Based on Proposition \ref{relat index}-$(c)$ in Appendix, we have
\bea\label{rela1}
I(\mathcal{A}_0,\mathcal{A}_0+\mathcal{B}_{\fe})=-I(\mathcal{B}_{\fe}^{-1},\mathcal{B}_{\fe}^{-1}+\mathcal{A}^{-1}_0).
\eea
From Proposition \ref{relat index}-$(a)$ in Appendix, we obtain
\bea\label{rela2}
I(\mathcal{A}_0,\mathcal{A}_0+\mathcal{B}_{\fe})=m^{-}(\mathcal{A}_0+\mathcal{B}_{\fe})-m^{-}(\mathcal{A}_0),\ \
I(\mathcal{B}_{\fe}^{-1},\mathcal{B}_{\fe}^{-1}+\mathcal{A}^{-1}_0)
=m^{-}(\mathcal{B}_{\fe}^{-1}+\mathcal{A}^{-1}_0)-m^{-}(\mathcal{B}_{\fe}^{-1}).
\eea
Formulas (\ref{rela1}) and (\ref{rela2}) together with the fact $\mathcal{B}_{\fe}$ is positive definite in $\bar{D}(\omega,2\pi)$, we have
\begin{equation}\label{omega morse index1}
m^{-}(\mathcal{A}_0+\mathcal{B}_{\fe})-m^{-}(\mathcal{A}_0)=-m^{-}(\mathcal{B}_{\fe}^{-1}+\mathcal{A}^{-1}_0).
\end{equation}
Since for $\beta\in(\beta_{2}(0,\omega),1)$,
\begin{eqnarray}
m^{-}(\mathcal{A}_{0})=2,\ \ m^{-}(\mathcal{A}_0+\mathcal{B}_{0})=2,\ \ m^{-}(\mathcal{A}_0+\mathcal{B}_{\fe})=1,\ \ \text{when}\ \ \fe\rightarrow1,\nonumber
\end{eqnarray}
we have
\begin{equation}
m^{-}(\mathcal{B}_{0}^{-1}+\mathcal{A}^{-1}_0)=0,\ \ m^{-}(\mathcal{B}_{\fe}^{-1}+\mathcal{A}^{-1}_0)=1,\ \ \text{when}\ \ \fe\rightarrow1,\nonumber
\end{equation}
Moreover, by Proposition \ref{lem;beta -1 to -3/4}, we know $\beta_{1}(\fe,\omega)<\beta_{1}(0,\omega)$, hence $m^{-}(\mathcal{A}_0+\mathcal{B}_{\fe})\geq1$ for $\beta\in(\beta_{2}(0,\omega),0),$ then (\ref{omega morse index1})
implies $m^{-}(\mathcal{B}_{\fe}^{-1}+\mathcal{A}^{-1}_0)\leq1$.
Therefore, for fixed $\beta\in(\beta_{2}(0,\omega),0)$, when $\fe$ goes from $0$ to $1$, the only eigenvalue of $\mathcal{B}_{\fe}^{-1}+\mathcal{A}^{-1}_0$ can cross the origin is the first eigenvalue $\lambda_{1}(\mathcal{B}_{\fe}^{-1}+\mathcal{A}^{-1}_0)$. We further assert that the first eigenvalue crosses the origin exactly once. The
reason is that the first eigenvalue can be expressed as
$$\begin{aligned}
	\lambda_1(\mathcal{B}_\fe^{-1}+\mathcal{A}_0^{-1})&=\lambda_1(\mathcal{C}+\fe\mathcal{D})=
\inf\frac{\left<(\mathcal{C}+\fe\mathcal{D})x,x\right>}{\left<x,x\right>},
\end{aligned}$$
where $\mathcal{C}=\mathcal{A}_0^{-1}+\frac{1}{|\beta|}$ and $\mathcal{D}=\frac{\cos\theta}{|\beta|}$, one can verify that it is a concave function with respect to $\fe$. Therefore, for fixed $\beta\in(\beta_{2}(0,\omega),0)$, as $\fe$ goes from $0$ to $1$, the first eigenvalue crosses the origin only once.
Since $\mathcal{A}_0+\mathcal{B}_{\fe}$ is non-degenerate iff $\mathcal{B}_{\fe}^{-1}+\mathcal{A}_0^{-1}$ is non-degenerate, it follows that for fixed
$\beta\in(\beta_{2}(0,\omega),0)$, as $\fe$ goes from $0$ to $1$, the degenerate curve
$\Gamma_2(\omega)$ intersects the line $\{(\beta,\fe): \fe\in[0,1)\}$ only once. This proves that $\beta_2(\fe,\omega)$ is increasing with respect to $\fe$.
\end{proof}
\begin{rem}\label{rem1}
We must mention that the above analysis does not work for degenerate curve $\Gamma_{1}(\omega), \omega=e^{2\pi i\nu}, \nu\in(0,1/2)$. The reason is that, after some similar
calculation, in domain $\bar{D}(\omega,2\pi)$, we have
$m^{-}(\mathcal{B}_{0}^{-1}+\mathcal{A}^{-1}_0)=2$, $m^{-}(\mathcal{B}_{\fe}^{-1}+\mathcal{A}^{-1}_0)=1$, for $\fe\rightarrow1^-,$
hence when $\fe$ goes from $0$ to $1$, the eigenvalue of $\mathcal{B}_{\fe}^{-1}+\mathcal{A}^{-1}_0$ passing through the origin is the second eigenvalue, we don't know
the concavity of the second eigenvalue with respect to $\fe$.
\end{rem}

\vskip 0.2 cm

\subsection{The proof of main theorem}\label{proof main thm}
From Proposition \ref{topo of energy surfaces}, we know that the energy surface $\mathfrak{M}(h,\varpi,\beta)$ is a compact and regular contact manifold, which is homeomorphic to a three-sphere $\mathbb{S}^{3}$ when $\beta>-1, -1<2h\varpi^2<0$. The existence of disk-like global surfaces of section has been proved by Theorem \ref{thm: GlobalSurface}.

Taking $\omega=e^{2\pi i\sqrt{1+\beta}}$, by direct computation, we know that $\beta=\beta_{1}(0,\omega)$ for $\beta\in(-1,-3/4)$, whereas $\beta=\beta_{2}(0,\omega)$ for $\beta\in(-3/4,0)$. 
From Proposition \ref{lem;beta -1 to -3/4}, $\beta_{1}(\fe,\omega)$ satisfies
$\beta_{1}(\fe,\omega)<\beta_1(0,\omega)$. By the monotonicity of the mean index with respect to $\beta$ given in Corollary \ref{mean index monotonicity}, we have
$
\hat{i}(\gamma_{\beta_{1}(\fe,\omega),\fe})<\hat{i}(\gamma_{\beta_{1}(0,\omega),\fe}).
$
Moreover, from (\ref{mean index 2}), the mean index is invariant along the degenerate curve $\beta_{1}(\fe,\omega)$, i.e.,
$
\hat{i}(\gamma_{\beta_{1}(\fe,\omega),\fe})=\hat{i}(\gamma_{\beta_{1}(0,\omega),0}).
$
So
$
2\sqrt{1+\beta_{1}(0,\omega)}=\hat{i}(\gamma_{\beta_{1}(0,\omega),0})=\hat{i}(\gamma_{\beta_{1}(\fe,\omega),\fe})<\hat{i}(\gamma_{\beta_{1}(0,\omega),\fe}),
$
which implies
$
2\sqrt{1+\beta}=\hat{i}(\gamma_{\beta,0})<\hat{i}(\gamma_{\beta,\fe})
$
for all $(\beta,\fe)\in(-1,-3/4)\times(0,1)$. 
Proposition \ref{pro;beta -3/4 to 0} implies $\beta_2(\fe,\omega)$ is increasing with respect to $\fe$, so
$\beta_{2}(\fe,\omega)>\beta_2(0,\omega)$. A similar analysis yields
$
2\sqrt{1+\beta}=\hat{i}(\gamma_{\beta,0})>\hat{i}(\gamma_{\beta,\fe})
$
for $(\beta,\fe)\in(-3/4,0)\times(0,1)$. Thus, for $\beta\in(-1,-3/4)\cup(-3/4,0)$ and $\fe\in(0,1)$, we have $2\sqrt{1+\beta}\neq\hat{i}(\gamma_{\beta,\fe})$. Since $\rho_{kep}=\hat{i}(\gamma_{\beta,\fe})/2$, it follows that $\rho_{kep}\neq\sqrt{1+\beta}$.

For $\beta=-3/4$, Proposition \ref{lem;beta -1 to -3/4} and \ref{pro;beta -3/4 to 0} show that $-3/4\in(\beta_{1}(\fe,-1),\beta_{2}(\fe,-1))$
for $\fe\in(0,1)$. From Theorem \ref{prop of degcurve}-$v)$, we have $\gamma_{\beta,\mathfrak{e}}(2\pi)\approx D(\lambda)$, which is hyperbolic, and the mean index remains unchanged. Similarly, for $\beta=0$, Theorem \ref{prop of degcurve}-$iii)$ gives $\gamma_{\beta,\mathfrak{e}}(2\pi)\approx I_{2}$, and the mean index is again invariant.
Therefore, $2\sqrt{1+\beta}=\hat{i}(\gamma_{\beta,0})=\hat{i}(\gamma_{\beta,\fe})$ for $\beta=-3/4$ or $0$, so $\rho_{kep}=\sqrt{1+\beta}\in \mathbb{Q}$. By Corollary \ref{infini peri orbit}, we conclude that there exist infinitely many periodic orbits on each compact and regular energy surface $\mathfrak{M}(h,\varpi,\beta)$ with $\beta\in(-1,0]$. This completes the proof of Theorem \ref{thm; exist infinitely many}.

\begin{appendices}\label{App}
\section{The Maslov-type index for symplectic paths}
We give a brief review of the Maslov-type index for symplectic paths, the relative Morse index and the spectral flow in this section. Details could be found in \cite{Lon02}.
Let $Sp(2n)$ be the set of $2n\times2n$ real symplectic matrices. For $\tau>0$, we consider paths in $Sp(2n)$:
\begin{equation} \label{sym path}
\mathcal{P}_{\tau}(2 n)=\left\{\gamma \in C([0, \tau], \operatorname{Sp}(2 n)) \mid \gamma(0)=I_{2 n}\right\}. \nonumber
\end{equation}
For any $\omega\in\mathbb{U}$(the unit circle in $\mathbb{C}$),
 the following  $\omega$-degenerate hypersurface of codimension one in
 $Sp(2n)$ is defined  \cite{Lon02}:
$$   Sp(2n)_\omega^0=\{M\in Sp(2n) | \det(M-\omega I_{2n})=0 \}.   $$

For $M\in Sp(2n)_\omega^0$, we define a co-orientation of  $Sp(2n)_\omega^0$
at $M$ by the positive direction $\frac{d}{dt}Me^{tJ}|_{t=0}$ of the path $Me^{tJ}$ with $|t|$ sufficiently small. We now give the definition of $\omega$-index \cite{Lon02}.

\begin{defi}\label{def:Maslov-type index} For $\omega\in \mathbb{U}$, $\gamma\in \mathcal{P}_{\tau}$, \textbf{the $\omega$-index} of $\gamma$ is defined as
$$i_\omega(\gamma)= \begin{cases}[e^{-\epsilon J}\gamma: Sp(2n)_\omega^0]-n, & \text { if } \omega=1 \\ [e^{-\epsilon J}\gamma: Sp(2n)_\omega^0], & \text { if } \omega\neq1,\end{cases}
 $$ where $\epsilon>0$ is small enough and $[\cdot : \cdot]$ stands for the intersection number. We also denote the mean index and the nullity of $\gamma$ by
 $$
 \hat{i}(\gamma)=\lim_{t\rightarrow+\infty}\frac{i_{1}(\gamma)|_{[0,t]}}{t}\tau,\ \ \nu_\omega(\gamma)=\dim_{\mathbb{C}}\ker_{\mathbb{C}}(\gamma(\tau)-\omega I).\ \
 $$
\end{defi}
Let $\Lambda_D:=\mathbb{R}\oplus\{0\}$ be a Lagrangian subspace of $\rr^2$. Assume that the path $\gamma\in \mathcal P_\tau(2n)$ is differentiable and satisfies $-J_{2n}\dot \gamma(t)\gamma(t)^{-1}|_{\Lambda_D}>0$ for every $t\in \mathbb R$. As in the Morse index theorem,
the mean index can be computed by
\begin{eqnarray}\label{mean index2}
\hat{i}(\gamma)=\lim_{k\to +\infty}\frac{1}{k}\sum_{0<t\leq k\tau} \operatorname{dim} \gamma(t) \Lambda_{D} \cap \Lambda_{D}.
\end{eqnarray}

 In \cite{Lon02}, Long provided an explicit formula for the mean index. First, every
 $M\in \Sp(2)$ is symplectically similar to one of the following normal forms:
\begin{equation}\label{definition of R(theta)}
D(\lambda)=\left(\begin{array}{cc}
\lambda & 0 \\
0 & \lambda^{-1}
\end{array}\right),\ \
N_{1}(\lambda, a)=\left(\begin{array}{cc}
\lambda & a \\
0 & \lambda
\end{array}\right), \ \
R(\vartheta)=\left(\begin{array}{cc}
\cos \vartheta & -\sin \vartheta \\
\sin \vartheta & \cos \vartheta
\end{array}\right),\nonumber
\end{equation}
where $\lambda\in \mathbb{R}\setminus\{0\}$, $a=\pm1,0$, and $\vartheta\in(0,\pi)\cup(\pi,2\pi)$. As a special case
 of Corollary $8.3.2$ in \cite{Lon02}, we have
\begin{equation} \label{mean index 2}
\hat{i}(\gamma)=\left\{\begin{array}{ll}
i_1(\gamma)-1+\vartheta/\pi, & \text {if}\ \ \gamma(\tau)\approx R(\vartheta), \vartheta \in(0, \pi) \cup(\pi, 2 \pi)  \\
i_1(\gamma)+1, & \text{if}\ \ \gamma(\tau)\approx N_1(1,a), a=1,0 \\ i_1(\gamma), & \text{other}\quad \text{cases.}
\end{array}\right.
\end{equation}

The Maslov-type index is essentially equal to the Morse index.
 Consider a second order system $\ddot{x}=D(t)x$, $t\in[0,\tau]$. Let $\mathcal{A}=-\frac{d^2}{dt^2}+D$, which is a self-adjoint operator on $L^2([0,\tau], \mathbb{C}^n)$ with domain
 $\bar{D}(\omega,\tau)=\{W^{2,2}([0,\tau], \mathbb{C}^n)|y(\tau)=\omega y(0), \dot{y}(\tau)=\omega \dot{y}(0)\}.
 $
 we define the nullity of $\mathcal{A}$ by $v_{\omega}(\mathcal{A}) =\dim\ker(\mathcal{A})$ and
denote by $\phi_\omega(\mathcal{A})$ its Morse index, then from \cite{Lon02}, we have
\begin{equation}\label{index equ}
\phi_{\omega}(\mathcal{A})=i_{\omega}(\gamma), \ \ v_{\omega}(\mathcal{A})=\nu_{\omega}(\gamma),
\end{equation}
where $\dot{\gamma}=J \left(\begin{array}{cc}
I_n & 0_n \\
0_n & -D(t)
\end{array}\right)\gamma, \gamma(0)=I.$ For general boundary conditions, we refer to \cite{HWY}.
\section{The relative Morse index and spectral flow}
In general, a relative Morse index for two
self-adjoint Fredholm operators $A$ and $F$ on Hilbert Space $\mathcal{H}$ is introduced via the spectral flow,
$
I(A,A+F)=-sf\{A_{s}\},
$
where $A_s$ is a path of self-adjoint Fredholm operator such that $A_{0}=A, A_{1}=A+F$, and $sf(A_{s})$ denotes the spectral flow of a path of self-adjoint Fredholm operators $\{A_s, s\in[0,1]\}$. We usually take $A_{s}=A+sF$. The spectral flow was first introduced by Atiyah, Patodi and Singer \cite{APS76} in their study of index theory on manifolds with boundary. Roughly speaking, the spectral flow of a path $\{A_{s},s\in [0,1]\}$ counts the net change in the number of negative eigenvalues of $A_{s}$ as $s$ goes from $0$ to $1$, where the enumeration follows from the rule that each negative eigenvalue crossing to the positive axis contributes $+1$ and each positive eigenvalue crossing to the negative axis contributes $-1$, and for each crossing the multiplicity of eigenvalue is counted. In the following, we list some useful properties of relative Morse index. More details can be found in \cite{LZ99}.
\begin{prop}\label{relat index}

(a) If the Morse index  $m^{-}(A+F), m^{-}(A)<+\infty$, we have
$$
I(A,A+F)=m^{-}(A+F)-m^{-}(A).
$$

(b) When $A$ is a closed self-adjoint invertible operator with
 compact resolvent and $F$ is bounded definite self-adjoint operator on $\mathcal{H}$, then
 $$
 I(A, A+F)=I(-F^{-1}, -F^{-1}-A^{-1}).
 $$

 (c) Under the condition in (b), if $A^{-1}+F^{-1}$ is non-degenerate on $\mathcal{H}$, then
 $$
 I(A, A+F)=-I(F^{-1}, F^{-1}+A^{-1}).
 $$
\end{prop}

\end{appendices}

\hfill\newline
\noindent{\bf Acknowledgement.}
 This work is partially supported by the National Key R\&D Program of China (2020YFA0713303). X.Hu and Z.Qiao are partially supported by the National Key R\&D Program of China NSFC(No. 12521001) and Taishan Scholars Climbing Program of Shandong(TSPD20240802). Y.Ou is partially supported by NSFC ($\sharp$12371192), the Young Taishan Scholars Program of Shandong Province ($\sharp$ tsqn202312055), and the Qilu Young Scholar Program of Shandong University.
 
 The authors are deeply grateful to Professor Lei Liu for his valuable discussion. They are also deeply grateful to the anonymous referees for their careful reading of the manuscript and insightful comments and suggestions.

\hfill\newline
\bibliographystyle{abbrv}
\bibliography{RefaniKep}

	\end{document}